\documentclass[12pt]{amsart}

\usepackage{color,graphicx,amssymb,latexsym,amsfonts,txfonts,amsmath,amsthm}
\usepackage{pdfsync,enumitem}
\usepackage{amsmath,amscd}
\usepackage[all,cmtip]{xy}

\usepackage{tikz}
\usetikzlibrary{matrix}

\input epsf

\usepackage{hyperref}
\hypersetup{
    colorlinks=true,       % false: boxed links; true: colored links
    linkcolor=blue,          % color of internal links
    citecolor=blue,        % color of links to bibliography
    filecolor=blue,      % color of file links
    urlcolor=blue           % color of external links
}

\newtheorem{theo}{Theorem}
\newtheorem{coro}{Corollary}
\newtheorem{prop}{Proposition}

\newtheorem{claim}{Claim}

\theoremstyle{remark}
\newtheorem{rema}{\bf Remark}
\newtheorem{example}{\bf Example}

\begin{document}

\title{Infinite-type Schottky groups and group actions on infinite-type surfaces}
\date{}

\author{Rub\'en A. Hidalgo}
\address{Departamento de Matem\'atica y Estad\'{\i}stica\\
Universidad de La Frontera, Francisco Salazar 01145, Temuco, Chile}
\email{ruben.hidalgo@ufrontera.cl}

\thanks{Partially supported by ANID Fondecyt Regular  Grant 1230001}


\begin{abstract}
We introduce a certain class of purely loxodromic free Kleinian groups, called infinite-type Schottky groups, which are defined by a suitable collection of simple loops on the Riemann sphere, in a similar way as in the case of Schottky groups of finite rank. An infinite-type Schottky group $\Gamma$ admits a $\Gamma$-invariant connected component $\Omega_{\Gamma}$ of its region of discontinuity $\Omega(\Gamma)$, such that every other connected component of $\Omega(\Gamma) \setminus \Omega_{\Gamma}$ is a topological disc with trivial $\Gamma$-stabilizer, and $\Omega_{\Gamma}/\Gamma$ is an infinite-type Riemann surface without planar ends. 

Let $F$ be a torsion-free purely hyperbolic Fuchsian group of the first kind such that $\Sigma_{F}={\mathbb H}^{2}/F$ is an infinite-type Riemann surface with no planar ends.
Then there exists an infinite-type Schottky group $\Gamma$ such that $\Sigma_{F}$ is isomorphic to $\Omega_{\Gamma}/F$ (retrosection theorem). 

If $G < {\rm Aut}(\Sigma_{F})$ acts freely and $\Sigma_{F}/G$ is of finite-type, then we observe that
(i) the existence of some infinite Schottky $\Gamma$ such that $\Omega_{\Gamma}/\Gamma$ and $\Sigma_{F}$ are conformally equivalent and for which $G$ lifts to a group of automorphisms of $\Omega_{\Gamma}$, is equivalent to (ii) the existence of a $G$-invariant collection ${\mathcal F}$ of pairwise disjoint essential simple loops on $\Sigma_{F}$ such that each connected component of $\Sigma_{F} \setminus {\mathcal F}$ is a finite planar surface. This generalizes the situation for the case of closed Riemann surfaces and Schottky groups of finite rank.

\end{abstract}

\maketitle

%%%%%%%%%%%%%%%%%%%%%%%%%%%%%
%%%%%%%%%%%%%%%%%%%%%%%%%%%%%
\section{Introduction}
Let $\Sigma$ be a connected hyperbolic Riemann surface without planar ends. So, either: (i) $\Sigma$ is a closed Riemann surface of some genus $g \geq 2$, or (ii) $\Sigma$ has infinite-type (i.e., its fundamental group is not finitely generated) such that its ends are accumulated by genus, for example, the Loch Ness monster, the Jacob ladder, and the Blooming three (for the classification of infinite type surfaces, see \cite{Ker, Ian}). The Riemann surface structure on $\Sigma$ is provided by a torsion-free purely hyperbolic Fuchsian group $F < {\rm PSL}_{2}({\mathbb R})$ such that there exists a regular covering $P_{F}:{\mathbb H}^{2} \to \Sigma$ with ${\rm Deck}(P_{F})=F$. The group $F$ is unique up to ${\rm PSL}_{2}({\mathbb R})$-conjugation. 

In this paper, we will assume that the Riemann surface $\Sigma$ is provided by a Fuchsian group $F$ of the {\it first kind}, i.e., its limit set is all of the the boundary of ${\mathbb H}^{2}$ (this is always the case if $\Sigma$ is a closed surface).

%%%%%%%%%%%%%%%
\subsection{Uniformizations}
An {\it uniformization} of $\Sigma$ is a triple $(\Omega,K,P:\Omega \to \Sigma)$, where $K$ is a Kleinian group (i.e., a discrete subgroup of ${\rm PSL}_{2}({\mathbb C})$, seen as the group of condormal automorphisms of the Riemann sphere $\widehat{\mathbb C}$), $\Omega$ is a non-empty $K$-invariant connected component of its region of discontinuity, and $P$ is a regular covering map with ${\rm Deck}(P)=\Gamma$. For instance, $({\mathbb H}^{2},F,P_{F}:{\mathbb H}^{2} \to \Sigma)$ is an uniformization, called a {\it Fuchsian uniformization}.

The collection of uniformizations of $\Sigma$ is partially ordered: $(\Omega_{2},K_{2},P_{2}:\Omega_{2} \to \Sigma) \leq (\Omega_{1},K_{1},P_{1}:\Omega_{1} \to \Sigma)$ if and only if there is a covering map $Q:\Omega_{1} \to \Omega_{2}$ such that $P_{1}=P_{2} \circ Q$. The highest uniformizations are those for which $\Omega$ is simply-connected (for instance, the case of Fuchsian uniformizations). We are interested in the lowest uniformizations of $\Sigma$.

%%%%%%%%%%%%%%
\subsection{Schottky system of loops}
A {\it Schottky system of loops} on $\Sigma$ is a collection ${\mathcal F}$ of pairwise disjoint essential simple loops on $\Sigma$, no two of them parallels, such that each connected component of $\Sigma \setminus {\mathcal F}$ is a finite planar surface (its boundary consists of a finite sub-collection of ${\mathcal F}$ and two boundaries might or not be defined by the same loop). Examples of such a system of loops are those provided by a pant decomposition. 
As every essential simple loop in $\Sigma$ is isotopic to a (unique) simple closed geodesic in the hyperbolic structure of $\Sigma$,  then (up to a suitable isotopy) we may change each simple loop in ${\mathcal F}$ by its unique isotopic geodesic loop. As we are assuming that $F$ is of the first kind,  the new collection (so obtained by geodesics) is also a Schottky system of loops in 
$\Sigma$. 
The  Schottky system of loops ${\mathcal F}$ determines a connected orientable $3$-manifold $M_{(\Sigma,{\mathcal F})}$, with boundary $\Sigma$. If $\Sigma$ is closed (respectively, of infinite type), this manifold is called a {\it handlebody} (respectively, an {\it infinite handlebody} \cite{BM}).

%%%%%%%%%%%%
\subsection{Schottky uniformizations of closed Riemann surfaces}
Let us first consider the well known case when $\Sigma$ is a closed surface of genus $g \geq 2$. In this case,  $F$ is a co-compact Fuchsian group, and the Schottky system of loops ${\mathcal F}$ determines a Schottky group $\Gamma$ of rank $g$, with a connected region of discontinuity $\Omega(\Gamma)$ (its limit set $\Lambda(\Gamma)$ is a Cantor set), and a {\it Schottky uniformization} $(\Omega(\Gamma),\Gamma,P_{\Gamma}:\Omega(\Gamma) \to \Sigma)$. This a lowest uniformization of $\Sigma$ and 
it is characterized by the property that $P_{\Gamma}^{-1}({\mathcal F})$ is a collection of pairwise disjoint simple loops inside $\Omega(\Gamma)$, and it is the highest uniformization with such a property. 
Moreover, inside $P_{\Gamma}^{-1}({\mathcal F})$ there is a sub-collection ${\mathfrak C}=\{C_{1},C'_{1},\ldots,C_{g},C'_{g}\}$, no one of these loops separating others two,  and there is a collection of loxodromic elements $A_{1},\ldots,A_{g} \in \Gamma$, generating $\Gamma$, such that $A_{i}({\rm Ext}(C_{i}))={\rm Int}(C'_{i})$ (where, for each $C \in  {\mathfrak C}$, ${\rm Int}(C)$ is the topological disc bounded by $C$ and disjoint from the other loops in ${\mathfrak C}$; the other disc bounded by $C$ is denoted by ${\rm Ext}(C)$). The collection ${\mathfrak C}$ projects to $\Sigma$ to a sub-collection ${\mathcal H}=\{\alpha_{1},\ldots,\alpha_{g}\} \subset {\mathcal F}$ such that $\Sigma \setminus {\mathcal H}$ is connected and planar.

Let us note that any two Schottky systems of loops on $\Sigma$ produce quasiconformally equivalent Schottky groups.

%%%%%%%%%%%%
\subsection{Schottky uniformizations of infinite-type Riemann surfaces}Now, let us assume $\Sigma$ is of infinite-type (recall that we are assuming the torsion-free and purely hyperbolic Fuchsian group $F$ to be of the first kind).
In Section \ref{Sec:infiniteSchottky} (see Theorem \ref{teoretro}), we observe that the Schottky system ${\mathcal F}$ defines a torsion-free purely loxodromic Kleinian groups $\Gamma$, admitting a $\Gamma$-invariant connected component $\Omega$ of its region of discontinuity $\Omega(\Gamma)$, such that:
(i) each connected component $\Delta$ of $\Omega(\Gamma) \setminus \Omega$ (if any) is a topological disk with trivial $\Gamma$-stabilizer, and (ii) there is an uniformization $(\Omega,\Gamma,P_{\Gamma}:\Omega \to \Sigma)$ of $\Sigma$; we call it an {\it infinite-type Schottky uniformization} of $\Sigma$.
  As for the previous case, $\Gamma$ is characterized by the property that $P_{\Gamma}^{-1}({\mathcal F})$ is a collection of pairwise disjoint simple loops. Moreover, inside such a collection of lifted loops there is a sub-collection ${\mathfrak C}=\{C_{i},C'_{i}\}_{i \in {\mathbb N}}$, such that (a) no one of these loops is separating others two, (b) each of of these loops is isolated from the others, (c) their (spherical) diameters go to zero, and (d)  there is a collection of loxodromic elements $A_{i} \in \Gamma$, $i \in {\mathbb N}$, generating $\Gamma$, such that $A_{i}({\rm Ext}(C_{i}))={\rm Int}(C'_{i})$. The collection ${\mathfrak C}$ projects to $\Sigma$ to a sub-collection ${\mathcal H}$ of ${\mathcal F}$ such that $\Sigma \setminus {\mathcal H}$ is connected and planar. 
 
In this infinite-type situation, there are different Schottky systems of loops on $\Sigma$ for which the induced infinite-type Schottky groups are not topologically conjugate. 
 
Let us remark that, in \cite{BM}, the authors constructed certain kind of infinitely generated Schottky groups by using 
 ``admissible configuration" of loops on the Riemann sphere. Their groups have the property that the limit sets are Cantor sets. The problem with these Schottky groups is (it seems??) that not every Riemann surface (uniformized by a Fuchsian group of the first kind) can be so uniformized (anyway, they prove that one may change the Riemann surface structure of $\Sigma$ to obtain a Schottky uniformization). In Section \ref{Sec:Schottkyinfinito},  following \cite{BM}, we recall the definition of these groups together with some of their properties.

To construct our groups, we will restrict to a collection of loops which we call Schottky admissible configuration (a particular class of admissible configuration). We will admit the possibility for the limit set not to be totally disconnected, but we ask for the existence of an invariant connected component of its region of discontinuity.

%%%%%%%%%%%%%%%%%%%%%%%%%%%%%%%%%%%%%%%
\subsection{Schottky uniformizations and groups of conformal automorphisms}
Let us consider a subgroup $G$ of the group ${\rm Aut}(\Sigma)$ of conformal automorphisms of $\Sigma$.
The hyperbolic metric of ${\mathbb H}^{2}$ induces a complete hyperbolic metric on $\Sigma$. The group $G$ turns out to be a group of isometries in such a metric.

We say that the group $G$ {\it lifts} with respect to a Schottky uniformization $(\Omega,\Gamma,P_{\Gamma}:\Omega \to \Sigma)$ of $\Sigma$ 
if for every $\varphi \in G$ there exists a $\psi \in {\rm Aut}(\Omega)$ such that $P_{\Gamma} \circ \psi = \varphi \circ P_{\Gamma}$. 

If there is a Schottky system ${\mathcal F}$ on $\Sigma$ which is $G$-invariant, then $G$ lifts with respect to the induced Schottky uniformization of $\Sigma$ for ${\mathcal F}$. Below, we discuss the converse.

\subsubsection{Case of closed Riemann surfaces}
In this case, $G$ is finite whose order is bounded above by the Hurwitz bound $84(g-1)$, and the existence of a Schottky uniformization for $\Sigma$ for which $G$ lifts is equivalent to the existence of a Schottky system ${\mathcal F}$ which is $G$-invariant.
This is a consequence of the equivariant loop theorem \cite{M-Y} (in \cite{H-M}, there is an argument of such a fact based on Kleinian groups).
Equivalently, the lifting property is the same as finding a handlebody $M=M_{(\Sigma,{\mathcal F})}$ for which the group $G$ extends continuously as an (isomorphic) group of homeomorphisms of $M$.
For instance, if $G$ is either an abelian group acting freely or a dihedral group, then it lifts to a suitable Schottky uniformization (equivalently, extends to some handlebody) \cite{H2, RZ}.

\subsubsection{Case of infinite-type Riemann surfaces}
In this case, $G$ might not be a finite group, but is a countable one.
As it was for the case for the finite-type situation, if there is a Schottky system ${\mathcal F}$ that is $G$-invariant, then $G$ lifts with respect to the infinite-type Schottky uniformization induced by ${\mathcal F}$.
In Theorem \ref{teolifting}, we prove that the lifting property of $G$ to some infinite-type Schottky uniformization of $\Sigma$  is equivalent to the existence of a $G$-invariant Schottky system under the assumption that $\Sigma/G$ is of finite type.

\smallskip

%%%%%%%%%%%%%
\subsection*{Notations}
The symbol $A \leq B$ (respectively, $A \lhd B$) will indicate that $A$ is a subgroup (respectively, normal subgroup) of the group $B$.

%%%%%%%%%%%%%%%%
%%%%%%%%%%%%%%%%
\section{Preliminaries}
%%%%%%%%%%%%%
\subsection{Kleinian groups}
A {\it Kleinian group} is a discrete subgroup of ${\rm PSL}_{2}({\mathbb C})$; the group of conformal automorphisms of the Riemann sphere $\widehat{\mathbb C}$ (also called the M\"obius group). 
If $K$ is a Kleinian group, then its region of discontinuity, denote by $\Omega(K)$, consists of those points $z \in \widehat{\mathbb C}$ such that there is an open set $U_{z}$, containing $z$, for which $A(U_{z}) \cap U_{z}=\emptyset$, with the possible exception of a finite collection of $A \in K$.
Its limit set is $\Lambda(K)=\widehat{\mathbb C} \setminus \Omega(K)$. If $\Omega(K) \neq \emptyset$, then a {\it fundamental domain} for $K$ is an open subset $F \subset \Omega(K)$ such that (i) every point of $\Omega(K)$ is in the $K$-orbit of a point in the closure of $F$, and (ii) no two different points in $F$ belong to the same $K$-orbit. 

In this paper, we will restrict to torsion-free Kleinian groups $K$, with non-empty region of discontinuity, admitting a $K$-invariant connected component. These groups might or not be finitely generated, and they are purely loxodromic. Details on Kleinian groups can be found, for instance, in the book \cite{Maskit:book}.

%%%%%%%%%%%%%
\subsection{Schottky admissible configurations} 
\subsubsection{}
Let ${\mathfrak C}=\{C_{i}\}_{i \in J}$, where $\emptyset \neq J \subseteqq {\mathbb N}$, be a collection of pairwise disjoint simple loops in the Riemann sphere $\widehat{\mathbb C}$.
We will say that ${\mathfrak C}$ a {\it Schottky admissible configuration} if the following three conditions hold.
\begin{enumerate}[leftmargin=15pt]
\item[(S1)] If $J$ is infinite, then the spherical diameters of these loops go to zero, in other words, for every $\varepsilon>0$ there exists a positive integer $N=N_{\varepsilon}$ such that, for every $i \geq N$, it holds that the spherical diameters of $C_{i}$ is less than $\varepsilon$.
\item[(S2)] No loop in ${\mathfrak C}$ separates any other two loops of ${\mathfrak C}$. 
\item[(S3)] No point on any of loops in ${\mathfrak C}$ is the limit of points from the other loops in ${\mathfrak C}$.
\end{enumerate}

Following \cite{BM}, if ${\mathfrak C}$ satisfies the conditions (S2) and (S3), then it is called an {\it admissible configuration}.

\begin{rema}\label{obs1}
\mbox{}
\begin{enumerate}[leftmargin=15pt]
\item If ${\mathfrak C}$ is a finite collection of pairwise disjoint loops, then conditions (S1) and (S3) are trivially satisfied. In particular, finite admissible configurations are also Schottky admissible.

\item There are (infinite) admissible configurations that are not  Schottky admissible (since the condition (S1) might not necessarily be satisfied). Any admissible configuration consisting solely of circles is a Schottky admissible one.

\item Condition (S3) ensures that, for each $C \in {\mathfrak C}$, and every point $z\in C$, there is an open disk $D_{z}$ containing $z$ such that 
$D_{z} \cap ({\mathfrak C}\setminus\{C\})=\emptyset$. In particular, 
there is an open set $U_{C}$ containing $C$ and such that $U_{C} \cap ({\mathfrak C}\setminus\{C\})=\emptyset$.

\item In \cite{BM}, the authors replace (S2) by the following equivalent one: (S2') there exists of a point $z_{0} \in \widehat{\mathbb C} \setminus {\mathfrak C}$ for which no loop in ${\mathfrak C}$ separates any other loop in ${\mathfrak C}$ from $z_{0}$. Such a point might or might not be accumulated by points in the loops in ${\mathfrak C}$ (but we may always choose it not to be accumulated by points in ${\mathfrak C}$ by taking it inside $U_{C_{1}} \setminus C_{1}$). 

\item If $A \in {\rm PSL}_{2}({\mathbb C})$ and ${\mathfrak C}=\{C_{i}\}_{i \in J}$ is an admissible configuration (respectively, a Schottky admissible configuration), then $A({\mathfrak C})=\{A(C_{i})\}_{i \in J}$ is also an admissible configuration (respectively, a Schottky admissible configuration). So, for (S2') there is no loss of generality in assuming that $z_{0}=\infty$ is not accumulated by points in ${\mathfrak C}$ (in this case, the admissible configuration is contained in a bounded domain of ${\mathbb C}$).
\end{enumerate}
\end{rema}

%%%%%%%%%%%%%%%%%
\subsubsection{}
Let ${\mathfrak C}$ be an (not necessarily Schottky) admissible configuration.
Each of the loops $C \in {\mathfrak C}$ bounds two open discs.  Condition (S2) asserts that one of these two discs, denoted as ${\rm Ext}(C)$, contains all other loops of ${\mathfrak C}$. We denote by ${\rm Int}(C)$ the other component. We set 
$${\rm Ext}({\mathfrak C}):=\bigcap_{C \in {\mathfrak C}} {\rm Ext}(C),$$
and let us denote as ${\rm Ext}({\mathfrak C})^{0}$ the interior of ${\rm Ext}({\mathfrak C})$.

We denote by ${\mathfrak C}'$ the collection of limits points of the loops in ${\mathfrak C}$, i.e., the set of points $p \in \widehat{\mathbb C}$ for which there exists an infinite sequence of different loops $E_{j} \in {\mathfrak C}$ and points $z_{j} \in  E_{j}$ such that the sequence $(z_{j})$ accumulates at $p$. In particular, ${\mathfrak C}' \subset {\rm Ext}({\mathfrak C})$.

\begin{rema}\mbox{}
\begin{enumerate}[leftmargin=15pt]
\item Condition (S3) ensures that ${\rm Ext}({\mathfrak C})^{0} \neq \emptyset$ (in particular, ${\rm Ext}({\mathfrak C}) \neq \emptyset$), but it might be that ${\rm Ext}({\mathfrak C})^{0}$ is non-connected. In particular, there are infinitely many points $z_{0} \in {\rm Ext}({\mathfrak C})^{0}$ satisfying condition (S2').
%\item The set ${\rm Ext}({\mathfrak C})$ is non-empty (as any point $z_{0}$, satisfying condition (S2') in Remark \ref{obs1}, belongs to ${\rm Ext}(C)$,  for every $C \in {\mathfrak C}$). 

\item ${\rm Ext}({\mathfrak C}) \setminus {\rm Ext}({\mathfrak C})^{0}= {\mathfrak C}'$. If ${\mathfrak C}$ is a finite collection, then ${\rm Ext}({\mathfrak C})= {\rm Ext}({\mathfrak C})^{0}$ and ${\mathfrak C}'=\emptyset$.

\item If ${\mathfrak C}=\{C_{i}\}_{i \in {\mathbb N}}$ is a Schottky admissible configuration, then condition (S1) ensures the following: 
if $p \in {\mathfrak C}'$, $E_{j} \in {\mathfrak C}$ is an infinite sequence of different loops, and $z_{j} \in  E_{j}$ are such that the sequence $(z_{j})$ accumulates at $p$, then for every choices of points $w_{j} \in {\rm Int}(E_{j})$, the sequence $(w_{j})$ also accumulates at $p$. This property does not hold in general for those admissible configurations that are not Schottky.

\item In \cite[Proposition A.6.]{Maskit:book} (see Remark\ref{ej:bernie}), there is constructed a collection ${\mathfrak C}=\{C_{i},C'_{i}\}_{i \in {\mathbb N}}$ of pairwise disjoint circles satisfying conditions (S1) and (S2), but not (S3); so it is not an admissible configuration. In there, using that configuration of circles, it is constructed a discrete group $\Gamma=\langle A_{1},\ldots\rangle$, where $A_{i}({\rm Ext}(C_{i}))={\rm Int}(C'_{i})$; which is in there called an infinite Schottky group. Below, we will define infinite Schottky groups which are obtained from Schottky admissible configurations, so ours are not the same as those.
\end{enumerate}
\end{rema}

%%%%%%%%%%%%%%%%%%
%\subsection{Schottky system of loops}
%Let $\Sigma$ be a connected, Hausdorff, second-countable, orientable surface with no planar ends (either of finite type or of infinite type).

%A {\it Schottky system of loops} on $\Sigma$ is a (necessarily countable) collection ${\mathcal F}$ of pairwise disjoint essential simple loops, no two of them being parallel, and such that each connected component of $\Sigma \setminus {\mathcal F}$ is a planar surface of finite type. 

%For example, each pant decomposition of $\Sigma$ determines a Schottky system of loops on it. 

%\begin{rema}
%If ${\mathcal F}$ is a Schottky system of loops of $\Sigma$, then it contains a sub-collection ${\mathcal H}$ such that $\Sigma \setminus {\mathcal H}$ is connected and planar. Such a sub-collection is not unique in general.
%\end{rema}

%\medskip

%Each Schottky system of loops ${\mathcal F}$ on $\Sigma$ determines an orientable $3$-manifold $M_{(\Sigma,{\mathcal F})}$ whose boundary is $\Sigma$, and there is a natural bijection between the ends of it and the ends of $\Sigma$. We call $M_{(\Sigma,{\mathcal F})}$ a $\Sigma$-handlebody.

%\begin{lemm}
%Let ${\mathcal F}_{1}$ and ${\mathcal F}_{2}$ be two Schottky system of loops on $\Sigma$. Then $M_{(\Sigma,{\mathcal F}_{1})}$ and $M_{(\Sigma,{\mathcal F}_{2})}$ are homeomorphic $3$-manifolds.
%\end{lemm}

%%%%%%%%%%%%%%%%
%%%%%%%%%%%%%%
\section{Finitely generated Schottky groups}
In this section, we recall a geometrical definition of Schottky groups of finite rank and some of their properties.  Let us fix an integer $g \geq 1$.

%%%%%%%%%%%%%%%%%%%%
\subsection{Schottky groups of rank $g$}
Let ${\mathfrak C}_{g}=\{C_{1},\ldots,C_{g}, C'_{1},\ldots, C'_{g}\}$ be a Schottky admissible configuration,   
and let $D_{g}={\rm Ext}({\mathfrak C})={\rm Ext}^{0}({\mathfrak C})$ be the corresponding domain, of connectivity $2g$, bounded by these loops.

If $A_{1},\ldots, A_{g} \in {\rm PSL}_{2}({\mathbb C})$ are loxodromic elements  such that $A_{i}({\rm Ext}(C_{i}))={\rm Int}(C'_{i})$, then 
the group 
$\Gamma=\langle A_{1},\ldots,A_{g}\rangle < {\rm PSL}_{2}({\mathbb R})$ is called a {\it Schottky group of rank $g$}. 
The set $\{A_{1},\ldots,A_{g}\}$ is called a {\it Schottky set of generators of $\Gamma$} associated to ${\mathfrak C}$.
We say that $\Gamma$ is a  {\it classical Schottky group} if it can be defined by using a Schottky admissible configuration formed solely by circles.

\begin{theo}[\cite{Maskit:Schottky}]
Every Schottky group $\Gamma$ is a purely loxodromic Kleinian group, isomorphic to a free group of rank $g$, with $D_{g}$ as a fundamental domain. If $g \geq 2$, its limit set $\Lambda(\Gamma)$ is a Cantor set (this limit set is the closure of the fixed points of the non-trivial elements of $\Gamma$). Its region of discontinuity $\Omega(\Gamma)=\widehat{\mathbb C} \setminus \Lambda(\Gamma)$ is connected and  $\Omega(\Gamma)/\Gamma$ is a closed Riemann surface of genus $g$. Moreover, $M_{\Gamma}=({\mathbb H}^{3} \cup \Omega(\Gamma))/\Gamma$ is a handlebody of genus $g$, i.e., homeomorphic to the connected sum of $g$ solid tori. Its interior has a complete hyperbolic metric and its conformal boundary is the Riemann surface $\Omega(\Gamma)/\Gamma$.
\end{theo}

\begin{rema}
 A Schottky group of rank $g$ may also be characterized as a finitely generated, purely loxodromic  Kleinian group with a non-empty region of discontinuity, which is isomorphic to a free group of rank $g$ \cite{Maskit:Schottky} (equivalently, a Schottky group of rank $g$ is a purely loxodromic geometrically finite Kleinian group isomorphic to free groups of rank $g$). 
\end{rema}

\begin{theo}[\cite{Bers:Schottky}, \cite{Chuckrow}]
\mbox{}
\begin{enumerate}[leftmargin=15pt]
\item Let $\Gamma$ be a Schottky group of rank $g$. If $\Gamma=\langle B_{1},\ldots, B_{g}\rangle$, then there is a Schoottky admissible configuration such that $\{B_{1},\ldots,B_{g}\}$ is a Schottky set of generators of $\Gamma$  associated to it.

\item If $\Gamma_{1}$ and $\Gamma_{2}$ are two Schottky groups of the same rank $g$, then there is a quasiconformal homeomorphism $\varphi:\widehat{\mathbb C} \to \widehat{\mathbb C}$ such that $\Gamma_{2}=\varphi \Gamma_{1} \varphi^{-1}$.
\end{enumerate}
\end{theo}

\begin{theo}[The retrosection theorem]
Let $\Sigma$ be a closed Riemann surface of genus $g \geq 1$, and ${\mathcal F}$ be a Schottky system of loops on $\Sigma$. 
Then there is a Schottky group $\Gamma$ of rank $g$ together an admissible Schottky admissible configuration ${\mathfrak C}$ for it, and a regular holomorphic covering $P:\Omega(\Gamma) \to \Sigma$, with Deck group $\Gamma$, such that ${\mathcal F}$ is a subcollection of $P({\mathfrak C})$.
\end{theo}

The retrosection theorem was first stated by Felix Klein in 1883 \cite{Klein} and proved rigorously by Koebe \cite{Koebe}. In \cite{Bers:Schottky}, Bers provided a proof by use of quasiconformal mappings.

\begin{rema}
Let $\Sigma$ be a closed Riemann surface of genus $g$. An {\it uniformization} of $\Sigma$ is a regular planar covering $Q:\Omega \to \Sigma$, with Deck group a Kleinian group $\Gamma$, where $\Omega$ is a $\Gamma$-invariant connected component of the region of discontinuity of $\Gamma$. In the case that $\Gamma$ is a Schottky group (in which case $\Omega=\Omega(\Gamma)$), we talk of a  {\it Schottky uniformization}. The retrocestion theorem states that $\Sigma$ admits Schottky uniformizations. The Schottky uniformizations of $\Sigma$ correspond to the smallest regular planar coverings of it for which a given Schottky system of loops on $\Sigma$ lifts to loops.
\end{rema}

\begin{rema}[Classical Schottky groups]
\mbox{}
\begin{enumerate}[leftmargin=15pt]
\item To check if a Schottky group $\Gamma$ is classical, one needs to find a Schottky admissible configuration of circles. A classical Schottky group admits Schottky admissible configurations that are not formed by circles.
In \cite{Marden}, Marden showed that, for every $g\geq 2$, there are non-classical Schottky groups of rank $g$; see also \cite{J-M-M:schottky}. An explicit family of examples of non-classical Schottky groups of rank two was constructed by Yamamoto in \cite{Yamamoto}, and a theoretical construction of an infinite collection of non-classical Schottky groups is provided in \cite{HM:Schottky}.

\item As there are non-classical Schottky groups, one may wonder if every closed Riemann surface $\Sigma$ of genus $g$ can or cannot be uniformized by a classical Schottky group of rank $g$. 
An affirmative answer was known in either one of the following two situations: (i) $\Sigma$ admits an anticonformal automorphism of order two with fixed points (Koebe), and (ii) $\Sigma$ has $g$ pairwise disjoint homologically independent short loops (McMullen). 
\end{enumerate}
\end{rema}

%%%%%%%%%%%%%%%%%%
\subsection{Schottky groups and Schottky system of loops}
Let $\Sigma$ be a closed Riemann surface of genus $g$.

By the retrosection theorem, there is a Schottky group $\Gamma$ of rank $g$ and a biholomorphism $\varphi:\Sigma \to \Omega(\Gamma)/\Gamma$. Let ${\mathfrak C}=\{C_{1},\ldots,C_{g}, C'_{1},\ldots,C'_{g}\}$ be a Schottky admissible configuration for $\Gamma$. The loops $C_{i}$'s project to a collection ${\mathcal G}=\{\alpha_{1},\ldots,\alpha_{g}\}$ of simple loops on $\Omega(\Gamma)/\Gamma$ such that $\Omega(\Gamma)/\Gamma \setminus {\mathcal G}$ is a connected planar domain of connectivity $2g$. So, ${\mathcal F}=\{\varphi^{-1}(\alpha),\ldots,\varphi^{-1}(\alpha_{g})\}$ is a Schottky system of loops on $\Sigma$.

Conversely, let ${\mathcal F}$ be any Schottky system of loops on $\Sigma$.
Inside ${\mathcal F}$ there is a sub-collection ${\mathcal H}$ which is also a Schottky system of loops on $\Sigma$ such that $\Sigma \setminus {\mathcal H}$ is a connected planar domain of connectivity $2g$. The collection ${\mathcal H}$ defines a Schottky group $\Gamma$ of rank $g$ together with a Schottky admissible configuration ${\mathfrak C}$ such that it reverses the above process. Moreover, every loop in ${\mathcal F}$ lifts to $\Omega(\Gamma)$ as a collection of simple loops.

%%%%%%%%%%%%%%%%%%%%
\subsection{A lifting/extension problem}
Let $\Sigma$ a closed Riemann surface of genus $g \geq 2$ and $G \leq {\rm Aut}(\Sigma)$ be a (necessarily finite) group of conformal automorphisms. 

Let $\Gamma$ be a Schottky group such that there is a Galois covering $P:\Omega(\Gamma) \to \Sigma$, with ${\rm Deck}(P)=\Gamma$ (i.e., a Schottky uniformization of $R$). We say that $G$ lifts under $P$ if for every $f \in G$ there is some $\widehat{f} \in {\rm PSL}_{2}({\mathbb C})$ such that $\widehat{f} \circ P=P \circ f$.

\medskip

The condition for $G$ to lift to some Schottky uniformization of $\Sigma$ is equivalent to saying that the group $G$ extends continuously as a (isomorphic) group of homeomorphisms of some handlebody $M$ with boundary $\Sigma$.

\medskip

As a consequence of the equivariant loop theorem \cite{M-Y} (see \cite{H-M} for proof of this fact in terms of Kleinian groups), the following hold.

\begin{theo}
Let $\Sigma$ be a closed Riemann surface of genus $g \geq 2$ and $G<{\rm Aut}(\Sigma)$.
Then there is a Schottky uniformization of $\Sigma$ for which the group $G$ lifts (equivalently, extends as a group of homeomorphisms to some handlebody) if and only if 
there is a $G$-invariant Schottky system of loops on $\Sigma$.
\end{theo}

The following is known (see, for instance, \cite{H1, H2,H4, RZ}).
\begin{enumerate}[leftmargin=15pt]
\item[(i)] If $\Sigma/G$ has genus zero and exactly three cone points, then it cannot be lifted to a Schottky uniformization. 
\item[(ii)] If $G$ acts freely (i.e., the $G$-stabilizer of every point of $\Sigma$ is trivial) and $\Sigma/G$ has genus one (so, $\Sigma$ also has genus one), then it can be lifted to a Schottky uniformization.
\item[(iii)] If $G$ acts freely and it is isomorphic to either an abelian group or one of the Platonic symmetry groups, then it can be lifted to a Schottky uniformization.
\item[(iv)] If $G$ is a dihedral group, then it can be lifted to a Schottky uniformization.
\end{enumerate}

\begin{rema}
Let $\Sigma$ be a closed orientable surface of genus $g$ (we do not need to assume it to be a Riemann surface).
Assume $G \leq {\rm Hom}^{+}(\Sigma)$ is finite and acting freely on $\Sigma$. 
The surface $\Sigma$ can be seen as the boundary of infinitely many compact orientable $3$-manifolds (for instance, handlebodies). 
As it is well known that  one can provide to $\Sigma$ of a Riemann surface structure such that $G$ acts a group of conformal automorphisms, 
the above theorem states a necessary and sufficient condition for $G$ to extend to a handlebody. 
In \cite{Samperton}, Samperton proved that if $B_{0}(G)=0$ (where $B(G)$ denotes the Bogomolov multiplier of $G$), then $G$ extends to some $3$-manifold $M$ (not necessarily a handlebody).
Examples of these types of groups are the abelian and dihedral groups (they extend to handlebodies). Other examples are the alternating and symmetric groups (see also \cite{DS}).
In the same paper, Samperton proved that if $B_{0}(G) \neq 0$ and $G$ does not contain a dihedral subgroup, then it admits a free action (on some surface) that does not extend to any $3$-manifold (for instance, $G={\rm SmallGroup}(3^{5},28)$ in the Gap Library \cite{GAP}). In \cite{H:extension}, it is proved the existence of finite groups $G$ acting freely as orientation-preserving homeomorphisms on closed orientable surfaces which extend as a group of homeomorphisms of some compact orientable $3$-manifold but which cannot extend to a handlebody. 
\end{rema}

%%%%%%%%%%%%%%%%
%%%%%%%%%%%%%%%%
\section{Infinitely generated Schottky groups}\label{Sec:Schottkyinfinito}
Above, we have recalled the notion of Schottky groups of finite rank. The natural question is how to define the notion of a Schottky group of infinite rank.
In \cite{BM}, Basmajian and Matsuzaki proposed a definition of weakly-Schottky groups and Schottky groups associated to admissible configurations. 
Below, we recall their definitions and some of the properties of these groups. We will name them as infinitely generated weakly-Schottky/Schottky groups.

%%%%%%%%%%%%%%%%%
\subsection{Infinitely generated Schottky-like groups}
Let ${\mathfrak C}=\{C_{i},C'_{i}\}_{i \in {\mathbb N}}$ be an admissible configuration. If 
$\{A_{i}\}_{i \in {\mathbb N}}$ is a collection of loxodromic elements such that $A_{i}({\rm Ext}(C_{i})) = {\rm Int}(C'_{i})$, then we call the 
group  $\Gamma=\langle A_{i}: i \in {\mathbb N}\rangle < {\rm PSL}_{2}({\mathbb C})$ an {\it infinitely generated Schottky-like group}. 
In the case that every element of ${\mathfrak C}$ is a circle, then we say that $\Gamma$ is a {\it classical} infinitely generated Schottky-like group. Let us denote by 
$\Omega(\Gamma)$ and $\Lambda(\Gamma)$ be the region of discontinuity and the limit set, respectively, of the group $\Gamma$.

\begin{rema}
In \cite{BM}, the authors require $\Gamma$ to be discrete. But, as ${\rm Ext}({\mathfrak C})^{0} \neq \emptyset$, discreteness always hold. \end{rema}

\begin{theo}[\cite{BM}]\label{teo2}
Let $\Gamma$ be an infinitely generated Schottky-like group defined by the admissible configuration ${\mathfrak C}$. Then
\begin{enumerate}[leftmargin=15pt]
\item $\Gamma$ is purely loxodromic and free on a countable number of generators.
\item ${\rm Ext}({\mathfrak C})$ is precisely invariant under the identity in $\Gamma$.
\item ${\rm Ext}({\mathfrak C})^{0} \cup {\mathfrak C} \subset \Omega(\Gamma)$, in particular, $\Gamma$ is discrete with $\Omega(\Gamma) \neq \emptyset$. 
\item If ${\mathfrak C}$ consists solely of circles (i.e., $\Gamma$ is classical), then the accumulation points of the configuration of circles form a closed subset of the limit set $\Lambda(\Gamma)$.
\end{enumerate}
\end{theo}
\begin{proof}
Parts (1) and (2) correspond to parts (1) and (2) on \cite[Proposition 2.1.]{BM}.
Part (3) follows from condition (S3). Part (4) is \cite[Lemma 2.2.]{BM}.
\end{proof}

\begin{rema}
If $\Gamma$ is an infinitely generated Schottky-like group defined by the admissible configuration ${\mathfrak C}$, then the following might happen.
\begin{enumerate}[leftmargin=15pt]
\item $\Lambda(\Gamma)$ might not be totally disconnected.
\item $\Omega(\Gamma)$ might not be connected, and it might not contain a $\Gamma$-invariant connected component.
\item There might be points in ${\mathfrak C}'$ that are in $\Omega(\Lambda)$ and others in $\Lambda(\Gamma)$.

\item ${\rm Ext}({\mathfrak C})^{0}$ might not be a fundamental domain for $\Gamma$. There might be points in ${\mathfrak C}'={\rm Ext}({\mathfrak C}) \setminus {\rm Ext}({\mathfrak C})^{0}$ belonging to $\Omega(\Gamma)$.
\end{enumerate}
\end{rema}

%%%%%%%%%%%%%%%%
\subsection{Infinitely generated Schottky groups}
Follwing \cite{BM}, if $\Gamma$ is an infinitely generated Schottky-like group, with associated admissible configuration ${\mathfrak C}$, such that (i) $\Lambda(\Gamma)$ is totally disconnected and (ii) ${\rm Ext}({\mathfrak C})^{0}$ is a fundamental domain for $\Gamma$,  then $\Gamma$ is called a {\it infinitely generated Schottky group}.

\begin{rema}[VIII.A. in \cite{Maskit:book}]\label{ej:bernie}
Let us consider an admissible configuration ${\mathfrak C}=\{C_{i},C'_{i}\}_{i \in {\mathbb N}}$ consisting solely of circles in ${\mathbb C}$. Let us assume that, for each $i$, the circles $C_{i}$ and $C'_{i}$ have the same radius. Let $m_{i}$ be the middle point of the arc-line $l_{i}$ connecting the centers of $C_{i}$ and $C'_{i}$. Let $\tau_{i}$ be the reflection on $C_{i}$ and let $\eta_{i}$ be the reflection on the orthogonal line to $l_{i}$ at $m_{i}$. The transformation $A_{i}=\eta_{i} \circ \tau_{i} \in {\rm PSL}_{2}({\mathbb C})$ is a loxodromic transformation whose isometric circle is $C_{i}$ (and $C'_{i}$ is the isometric circle of $A_{i}^{-1}$). Let us consider the infinitely generated Schottky-like group $\Gamma=\langle A_{1},\ldots\rangle$. As a consequence of \cite[Proposition A.4.]{Maskit:book}, ${\rm Ext}({\mathfrak C})^{0}$ is a fundamental domain for $\Gamma$. But, it might be that its limit set is not totally disconnected.
\end{rema}

%%%%%%%%%%%%%
\subsubsection{\bf Topological Retrosection Theorem}
If $\Gamma$ is an infinitely generated Schottky group, then $\Omega(\Gamma)/\Gamma$ is an infinite-type Riemann surface without planar ends.
The following result provides a (topological) converse.

\begin{theo}[Topological retrosection theorem \cite{BM}]
Let $\Sigma$ be a Riemann surface of infinite type and without planar ends. Then 
\begin{enumerate}[leftmargin=15pt]
\item There is an infinitely generated classical Schottky group $\Gamma$ such that $\Omega(\Gamma)/\Gamma$ is homeomorphic to $\Sigma$. 

\item If $\Sigma$ has a bounded pants decomposition, then there is an infinitely generated Schottky group $\Gamma$, which is quasiconformally conjugated to an infinitely generated classical Schottky group,
such that $\Sigma$ and $\Omega(\Gamma)/\Gamma$ are conformally equivalent. 
\end{enumerate}
\end{theo}

Let $\Sigma={\mathbb H}^{2}/F$ be an infinite-type Riemann surface without planar ends, where $F$ is an infinitely generated torsion-free Fuchsian group. 
The above theorem asserts that, if $\Sigma$ satisfies some properties (for instance, with a bounded pant decomposition), then it can be uniformized by an infinitely generated Schottky group. 
Two things might happen: either $F$ is of the first kind, or it is of the second kind. In \cite{BM}, the authors observed the following fact.

\begin{theo}[\cite{BM}]
Let $\Sigma={\mathbb H}^{2}/F$ be an infinite Riemann surface without planar ends. If $\Sigma$ admits a Schottky uniformization, then $F$ is of the first kind.
\end{theo}

%\begin{rema}
%The above result asserts that to have a Schottky uniformization of ${\mathbb H}^{2}/F$, it is necessary for $F$ to be of the first kind. Is the converse true?
%\end{rema}

%%%%%%%%%%%%%
\subsubsection{\bf Schottky uniformizations of infinite handlebodies}
Let $\Gamma$ be an infinitely generated Schottky group with admissible configuration ${\mathfrak C}$.
One may consider the $3$-manifold $M_{\Gamma}=({\mathbb H}^{3}  \cup \Omega(\Gamma))/\Gamma$, whose conformal boundary is the infinite type Riemann surface $\Sigma=\Omega(\Gamma)/\Gamma$.
The admissible configuration ${\mathfrak C}$ induces a Schottky system ${\mathcal F}$ on $\Sigma$, and this Schottky system defines the infinite handlebodiy $M_{(\Sigma,{\mathcal F})}$.   A natural question is if $M_{\Gamma}$ and $M_{(\Sigma,{\mathcal F})}$ are homeomorphic. An answer to this question was stated in \cite{BM}.

\begin{theo}[\cite{BM}]\label{teo6}
Let $\Gamma$ be an infinitely generated Schottky group with admissible configuration ${\mathfrak C}=\{C_{i},C´_{i}\}_{i \in {\mathbb N}}$. 
\begin{enumerate}[leftmargin=15pt]
\item Assume that the configuration
 satisfies that for any subsequence $\{i_{k}\} \subset {\mathbb N}$
 $$
(*) \quad {\rm if } \; C_{i_{k}} \to x \; {\rm then } \; C'_{i_{k}} \to x, \; {\rm as } \; k \to \infty.
$$
 Then $M_{\Gamma}$ and $M_{(\Sigma,{\mathcal F})}$ are homeomorphic. Conversely, any infinite handlebody can be topologically
uniformized by such a classical Schottky group.

\item Let $\Gamma$ be an infinitely generated classical Schottky group such that $M_{\Gamma}$ is an infinite handlebody. Then there exists a homeomorphism between the space of ends of $M_{\Gamma}:=({\mathbb H}^{3} \cup \Omega(\Gamma))/\Gamma$ and the space of ends of $\Omega(\Gamma)/\Gamma$. 

\end{enumerate}
\end{theo}

%%%%%%%%%%%%%%%%%%%%%%%%%%%
%%%%%%%%%%%%%%%%%%%%%%%%%%%
\section{Infinite-type Schottky groups}\label{Sec:infiniteSchottky}
In this section, we define a class of Kleinian groups which we call infinite-type Schottky groups. These are of a different type from those in the previous section. In our case, the limit set might not be totally disconnected, and the region of discontinuity might not be connected (but there will always be a (unique) invariant connected component).

%%%%%%%%%%%%%%%
\subsection{Infinite-type weakly-Schottky groups}
Let ${\mathfrak C}=\{C_{i},C'_{i}\}_{i \in {\mathbb N}}$ be a Schottky admissible configuration, and let ${\mathfrak C}'$ be the set of limit points of ${\mathfrak C}$, i.e., those points $p \in \widehat{\mathbb C}$ such that there exists an infinite collection of pairwise different loops $E_{1}, E_{2}, \ldots \in {\mathfrak C}$ and points $x_{i} \in E_{i}$ for which $p$ is accumulated by the sequence $(x_{i})$.

If $\{A_{i}\}_{i \in {\mathbb N}}$ is a collection of loxodromic elements such that $A_{i}({\rm Ext}(C_{i})) = {\rm Int}(C'_{i})$, then the corresponding infinitely generated Schottky-like group $\Gamma=\langle A_{i}: i \in {\mathbb N}\rangle < {\rm PSL}_{2}({\mathbb C})$ is called an {\it infinite-type weakly-Schottky group}.

Moreover, if all the loops in the Schottky admissible system  ${\mathfrak C}$ are circles, then we say that $\Gamma$ is a {\it classical} infinite-type weakly-Schottky group.

If $X \subset \widehat{\mathbb C}$, then $\Gamma(X)$ denotes the union of the sets $A(X)$, where $A \in \Gamma$.

\begin{theo}\label{teo3}
Let $\Gamma$ be an infinite-type weakly-Schottky group defined by the Schottky admissible configuration ${\mathfrak C}$. Then
\begin{enumerate}[leftmargin=15pt]
\item ${\rm Ext}({\mathfrak C})^{0} \cup {\mathfrak C} \subset \Omega(\Gamma)$, in particular,  $\Gamma$ is discrete with $\Omega(\Gamma) \neq \emptyset$, 
\item $\Gamma$ is purely loxodromic and free on a countable number of generators, 
\item ${\rm Ext}({\mathfrak C})$ is precisely invariant under the identity in $\Gamma$, and 
\item $\Gamma({\rm Ext}({\mathfrak C})^{0} \cup {\mathfrak C}) = \Omega(\Gamma)$, in particular,  ${\rm Ext}({\mathfrak C})^{0}$ is a fundamental domain for $\Gamma$. 
\item ${\mathfrak C}' \subset \Lambda(\Gamma)$.
\item If ${\mathfrak C}'$ is totally disconnected, then $\Lambda(\Gamma)$ is totally disconnected. In particular, $\Gamma$ is an infinitely generated Schottky group.

\end{enumerate}
\end{theo}
\begin{proof}
Parts (1),  (2), and (3) follow from Theorem \ref{teo2} as infinite-type weakly-Schottky groups are particular examples of infinitely generated Schottky-like groups. 
Let us now check part (4). We know (by part (1)) that $\Gamma({\rm Ext}({\mathfrak C})^{0} \cup {\mathfrak C}) \subset \Omega(\Gamma)$. 
Let $p \in \widehat{\mathbb C} \setminus \Gamma({\rm Ext}({\mathfrak C})^{0} \cup {\mathfrak C})$. Then, there is a collection of different loops $E_{1},E_{2}, \ldots \in {\mathfrak C}$ and elements $T_{i} \in \Gamma$, such that such that $p \in T_{i}({\rm Int}(E_{i}))$. But, as a consequence of the condition (S1), the point $p$ is then in the limit set of $\Gamma$.
This also ensures that ${\rm Ext}({\mathfrak C})^{0}$ is a fundamental domain for $\Gamma$.
Part (5) follows from condition (S1).
Part (6) follows from the fact that $\Lambda(\Gamma) \setminus {\mathfrak C}'$ is totally disconnected.
\end{proof}

\begin{rema}
\mbox{}
\begin{enumerate}[leftmargin=15pt]
\item For an infinitely generated Schottky-like group $\Gamma$, with associated admissible configuration ${\mathfrak C}$,  there might be that some points in ${\mathfrak C}'$ belong to $\Omega(\Gamma)$. The above result asserts that this is not the case for the case when ${\mathfrak C}$ is a Schottky admissible configuration.

\item If $\Gamma$ is an infinite-type weakly-Schottky group, then (i) $\Omega(\Gamma)$ is not necessarily connected, (ii) there might not be a $\Gamma$-invariant connected component of $\Omega$, 
and (iii) if $\Omega(\Gamma)$ is connected, then its limit set does not need to be a Cantor set.
\end{enumerate}
\end{rema}

\begin{example}
In this example, we describe an infinite-type weakly-Schottky group $\Gamma$, with a non-connected region of discontinuity $\Omega(\Gamma)$, such that there is no $\Gamma$-invariant component of $\Omega(\Gamma)$.
Let us consider the unit circle $C=\{z \in {\mathbb C}: |z|=1\}$ and the reflection $\tau(z)=1/\overline{z}$. Inside the unit disk ${\mathbb D}=\{z \in {\mathbb C}: |z|<1\}$ consider a collection of pairwise disjoint closed discs $D_{1}, D_{2},\ldots$ such that they accumulate at $C$. Let $C_{i}$ be the boundary circle of $D_{i}$ and $\tau_{i}$ the reflection whose locus of fixed points is $C_{i}$. Set $C'_{i}=\tau(C_{i})$ and $A_{i}=\tau \circ \tau_{i}$. Then $\Gamma=\langle A_{1},\ldots\rangle$ is an infinite-type weakly-Schottky group. Let us observe that in this case $C \subset \Lambda(\Gamma)$. If $\Delta$ is a connected component of $\Omega(\Gamma)$, then it is contained in either of the two open discs  ${\mathbb D}$ and $\widehat{\mathbb C} \setminus ({\mathbb D} \cup C)$. As $A_{1}$ (respectively, $A_{1}^{-1}$) sends the first (respectively, second) disk to the second (respectively, first) disk, $\Delta$ cannot be $\Gamma$-invariant.
\end{example}

%%%%%%%%%%%%%%%%%%
\subsection{\bf Infinite-type Schottky groups}
An infinite-type weakly-Schottky group $\Gamma$ is called an {\it infinite-type Schottky group} if $\Omega(\Gamma)$ has a $\Gamma$-invariant connected component $\Omega$.

As $\Gamma$ is a non-elementary Kleinian group, it contains at most two $\Gamma$-invariant connected components \cite[Proposition E.9.]{Maskit:book}.

\begin{prop}
Let $\Gamma$ be an infinite-type Schottky group, with $\Gamma$-invariant connected component $\Omega \subset \Omega(\Gamma)$. If $\Omega(\Gamma) \setminus \Omega \neq \emptyset$, then each of its connected components $\Delta$ is a topological disk and has a trivial $\Gamma$-stabilizer.
\end{prop}
\begin{proof}
Let $\Delta$ be a connected component of $\Omega(\Gamma) \setminus \Omega$  (if any).
Since $\Omega \cup \Lambda(\Gamma)$ is connected, it follows that $\Delta$ is a topological disk.

Let ${\mathfrak C}=\{C_{i},C'_{i}\}_{i \in {\mathbb N}}$ be a Schottky admissible configuration for $\Gamma$ with respect to the loxodromic generators $\{A_{i}\}_{i \in {\mathbb N}}$, where $A_{i}({\rm Ext}(C_{i})) = {\rm Int}(C'_{i})$.

There must be an index $j_{0} \in {\mathbb N}$ such that $\Delta \subset {\rm Int}(C_{j_{0}})$ or $\Delta \subset {\rm Int}(C'_{j_{0}})$. Without loss of generality, may assume that 
$\Delta \subset {\rm Int}(C_{j_{0}})$ (otherwise, permute the loops $C_{j_{0}}$ and $C'_{j_{0}}$, and change $A_{j_{0}}$ by its inverse).

Then, for every $i \neq j_{0}$, $A_{i}(\Delta) \subset {\rm Int}(C'_{i})$, i.e., $A_{i}(\Delta) \cap \Delta=\emptyset$. So, following a ping-pong argument, if $A \in \Gamma$ is not in $\langle A_{j_{0}}\rangle$, then $A(\Delta) \cap \Delta=\emptyset$.

As $\Delta \subset {\rm Int}(C_{j_{0}})$ is a non-empty open set, there is some positive integer $n$ such that $A_{j_{0}}^{n}(\Delta) \cap {\rm Ext}({\mathfrak C})^{0} \neq \emptyset$, again a contradiction.
\end{proof}

\begin{theo}[Retrosection theorem]\label{teoretro}
Let $\Sigma$ be an infinite-type Riemann surface without planar ends uniformized by a Fuchsian group of the first kind. Let ${\mathcal F}$ be a Schottky system of loops on $\Sigma$.
Then there is an infinite-type Schottky group $\Gamma$, with an invariant connected component  $\Omega \subset \Omega(\Gamma)$, 
together an admissible Schottky admissible configuration ${\mathfrak C}$ for it, and a regular holomorphic covering $P:\Omega(\Gamma) \to \Sigma$, with Deck group $\Gamma$, such that ${\mathcal F}$ is a subcollection of $P({\mathfrak C})$.
\end{theo}
\begin{proof}
This will be done later using the $(\Sigma,{\mathcal F})$-Schottky groups to be defined laterr.
\end{proof}

\begin{rema}
In the above retrosection theorem, we have imposed that the uniformizing Fuchsian group is of the first kind. But there are infinite-type Schottky groups $\Gamma$ with $\Omega=\Omega(\Gamma)$ such that $\Omega/\Gamma$ is uniformized by a Fuchsian group of the second kind. For instance, take as ${\mathfrak C}'$ the closed real interval $[0,1]$ and as ${\mathfrak C}$ a collection of circles accumulating to the points on $[0,1]$ from the upper half-plane. In this case, we can find half-planes from the lower half-plane with $[0,1]$ as part of its boundary.
\end{rema}

\begin{rema}
Let  $\Sigma$ be an infinite-type Riemann surface with no planar ends. Let $\Gamma$ be an infinite-type Schottky group with $\Gamma$-invariant component $\Omega$, and a regular covering $P:\Omega \to \Sigma$, with ${\rm Deck}(P)=\Gamma$. Let ${\mathfrak C}$ be a Schottky admissible configuration for $\Gamma$.
Then the following facts hold.
\begin{enumerate}[leftmargin=15pt]
\item The collection $P({\mathfrak C}) \subset \Sigma$ is a Schottky system of loops on $\Sigma$. As we will see in the proof of Theorem \ref{teoretro}, every Schottky system of loops on $\Sigma$ is obtained in the above way. 

\item Each point in ${\mathfrak C}'$ induces an end of $\Sigma$ and all ends of $\Sigma$ are so obtained. It might be that two different points in ${\mathfrak C}'$ induce the same end. 
%If the configuration ${\mathfrak C}$ satisfies the condition (*) in part (1) of Theorem \ref{teo6}, then different points of ${\mathfrak C}'$ induce different ends of $\Sigma$.
\end{enumerate}
\end{rema}

\begin{theo}
Let $\Gamma$ be an infinite-type Schottky group with $\Gamma$-invariant component $\Omega$ and with admissible configuration ${\mathfrak C}=\{C_{i},C´_{i}\}_{i \in {\mathbb N}}$. 
Let us denote by ${\mathfrak L}:=\{L_{j}:j \in J\}$ the set of connected components of ${\mathfrak C}'$.
Assume that the configuration
 satisfies that for any subsequence $\{i_{k}\} \subset {\mathbb N}$
 $$
{(**) \quad \rm if } \; C_{i_{k}} \to x \in L_{j}, \; \mbox{\rm for some $j \in J$,  then } \; C'_{i_{k}} \to x' \in L_{j}, \; {\rm as } \; k \to \infty.
$$
Then there exists a homeomorphism between the space of ends of $\Omega/\Gamma$ and the space of ends of $\Omega$. 
\end{theo}
\begin{proof}
The proof follows the same arguments as for the proof of \cite[Lemma 3.4]{BM}, and the fact that the space of ends of $\Omega$ is ${\mathfrak L}$.
\end{proof}

\begin{rema}
Let $\Gamma$ be an infinite-type Schottky group with $\Gamma$-invariant component $\Omega$ and with admissible configuration ${\mathfrak C}$.   If $\Omega(\Gamma) \neq \Omega$, then $M_{\Gamma}:=({\mathbb H}^{3} \cup \Omega(\Gamma))/\Gamma$ is an orientable $3$-manifold whose conformal boundary consits of $\Sigma:=\Omega/\Gamma$ and some copies of the open unit disc (these come from those connected compoents of $\Omega(\Gamma) \neq \Omega$). In the case that $\Omega(\Gamma)=\Omega$, then condition (**) in the above theorem permits to ensure that $M_{\Gamma}$ is homeomorphic to $M_{\Sigma,{\mathcal F}}$ for ${\mathcal F}$ a Schottky system of loops induced by ${\mathfrak C}$.
\end{rema}

\begin{theo}
Let $\Gamma$ be an infinite-type Schottky group with a connected region of discontinuity. Then the fundamental domain ${\rm Ext}({\mathfrak C})^{0}$ for $\Gamma$ is connected.
\end{theo}
\begin{proof}
Since $\Omega(\Gamma)$ is connected, it follows that ${\mathfrak C}'$ cannot contain a simple loop whose bounded topological discs contain points of $\Omega(\Gamma)$; in particular, it cannot disconnect ${\rm Ext}({\mathfrak C})$. ??????
\end{proof}

\section{Uniformizations by infinite-type Schottky groups: proof of Theorem \ref{teoretro}}\label{Sec:pruebateoretro}
Let $\Sigma={\mathbb H}^{2}/F$ be an infinite-type Riemann surface with no planar ends,  where $F$ is a Fuchsian group of the first kind. 
Let $\pi_{F}:{\mathbb H}^{2} \to \Sigma$ be a Galois covering with ${\rm Deck}(\pi_{F})=F$.

%%%%%%%%%%%%%%
\subsection{$(\Sigma,{\mathcal F})$-Schottky groups} \label{Schottkyinvariante}
Let ${\mathcal F}$ be a Schottky system of loops on $\Sigma$.

\subsubsection{}
The Riemann surface $\Sigma$ has a complete hyperbolic structure induced from that of ${\mathbb H}^{2}$ via $\pi_{F}$.
Each loop in ${\mathcal F}$ is isotopic to a unique closed geodesic. So, we may assume that all the loops in ${\mathcal F}$ are simple closed geodesics.

\subsubsection{}
If $\alpha \in {\mathcal F}$, then each connected component $\delta$ of $\pi_{F}^{-1}(\alpha)$ is a geodesic arc in ${\mathbb H}^{2}$. 
Let $\widetilde{\mathcal F}$ be the collection of all the geodesic arcs obtained by lifting all the geodesic loops of ${\mathcal F}$. This is a collection of pairwise disjoint geodesic arcs. Any two different geodesic arcs in $\widetilde{\mathcal F}$ have disjoint endpoints in the boundary.

\subsubsection{}
If $\delta \in \widetilde{\mathcal F}$, then its $F$-stabilizer is a cyclic group generated by a primitive hyperbolic element $A_{\delta} \in F$. If $\delta_{1}$ and $\delta_{2}$ are connected componts of $\pi_{F}^{-1}(\alpha)$, for $\alpha \in {\mathcal F}$, then $A_{\delta_{1}}$ is $F$-conjugated to either $A_{\delta_{2}}$ or $A_{\delta_{2}}^{-1}$.

\subsubsection{}
Let us denote by $N=N_{\mathcal F}$ the subgroup of $F$ generated by all the hyperbolic elements $A_{\delta}$, where $\delta \in \widetilde{\mathcal F}$. The subgroup $N_{\mathcal F}$ is a non-trivial normal subgroup of $F$, so a Fuchsian group of the first kind.

\subsubsection{}
The condition for ${\mathcal F}$ to be a Schottky system of loops ensures that $\Omega={\mathbb H}^{2}/N_{\mathcal F}$ is a planar Riemann surface with $\Gamma_{\mathcal F}=F/N_{\mathcal F}$ (a free group of infinite rank) as a group of conformal automorphisms such that $\Omega/\Gamma_{\mathcal F}=\Sigma$.

\subsubsection{}
Let $\pi_{N_{\mathcal F}}:{\mathbb H}^{2} \to \Omega$ be a Galois covering with ${\rm Deck}(\pi_{N_{\mathcal F}})=N_{\mathcal F}$, and 
$\pi_{\Gamma_{\mathcal F}}:\Omega \to \Sigma$ be a Galois covering with ${\rm Deck}(\pi_{\Gamma_{\mathcal F}})=\Gamma_{\mathcal F}$.

\subsubsection{}
By \cite[Thm 2]{Haas}, we may think of $\Omega$ as a domain inside the Riemann sphere $\widehat{\mathbb C}$ and $\Gamma_{\mathcal F} \leq {\rm Aut}(\Omega) \leq {\rm PSL}_{2}({\mathbb C})$. Moreover, from 
\cite[Thm 7]{Haas}, we may also assume that $\Omega$ is a pseudocircle domain (i.e., its boundary consists of circles, single point boundary components, and limit boundary components \cite[Section 4]{Haas}). The domain $\Omega$ is a connected component of the region of discontinuity of $\Gamma_{\mathcal F}$.

\subsubsection{}
Each loop $\alpha \in {\mathcal F}$ lifts, under $\pi_{\Gamma_{\mathcal F}}$, to a collection of simple loops in $\Omega$. Let us denote by  ${\mathcal G} \subset \Omega$ such a collection of simple loops.
Note that ${\mathcal G}=\pi_{N_{\mathcal F}}(\widetilde{\mathcal F})$.

If $\beta \in {\mathcal G}$, then $\pi_{\Gamma_{\mathcal F}}:\beta \to \alpha=\pi_{\Gamma_{\mathcal F}}(\beta)$ is a homeomorphism, and for each connected component $Y$ of $\Omega \setminus {\mathcal G}$ the restriction $\pi_{\Gamma_{\mathcal F}}:Y \to X=\pi_{\Gamma_{\mathcal F}}(Y)$ is also a homeomorphism (note that $X$ is a connected component of $\Sigma \setminus {\mathcal F}$).

\subsubsection{}
Let us observe that it might be that $\widehat{\mathbb C} \setminus \Omega$ has a non-empty interior, but its boundary is the limit set of $\Gamma_{\mathcal F}$. Each connected component of $\Omega(\Gamma_{\mathcal F}) \setminus \Omega$ is a topological disk with trivial $\Gamma_{\mathcal F}$-stabilizer. We call $\Gamma_{\mathcal F}$ a $(\Sigma,{\mathcal F})$-Schottky group.

\begin{theo}
Every $(\Sigma,{\mathcal F})$-Schottky group is an infinite-type Schottky group.
\end{theo}
\begin{proof}
Take a subcollection ${\mathcal F}'$ of ${\mathcal F}$ such that $\Sigma \setminus {\mathcal F}'$ is a domain. Lift this collection, under $\pi_{F}$, to obtain a subcollection $\widetilde{\mathcal F}'$ of $\widetilde{\mathcal F}$. Now consider the subcollection $\pi_{N_{\mathcal F}}(\widetilde{\mathcal F}')$ of ${\mathcal G}$. Inside $\pi_{N_{\mathcal F}}(\widetilde{\mathcal F}')$ there is a Schottky admissible configuration for $\Gamma$.
\end{proof}

As a consequence, we obtain the desired result.

\begin{coro}
Every Riemann surface $\Sigma$ of infinite-type without planar ends and uniformized by a Fuchsian group of the first kind can be uniformized by a $(\Sigma,{\mathcal F})$-Schottky group $\Gamma$. Moreover, $\Gamma$  can be chosen so that ${\rm Aut}(\Omega(\Gamma) \leq {\rm PSL}_{2}({\mathbb C})$.
\end{coro}

%{\color{red} How to ensure  that we may chose $\Gamma$ such that $\widehat{\mathbb C} \setminus \Omega$ is a Cantor set ????. This is needed to obtain a Schottky group of infinite rank.} 

%%%%%%%%%%%%%%%
%%%%%%%%%%%%%%%
\section{A lifting problem for infinite-type Schottky groups}
Let $\Sigma$ be an infinite-type surface (not necessarily a Riemann surface) without planar ends and let $G < {\rm Hom}^{+}(\Sigma)$ act freely and properly discontinuously. 
So, the quotient $\Sigma/G$ is a surface, and there is a Galois covering $Q:\Sigma \to \Sigma/G$ with ${\rm Deck}(Q)=G$.

%%%%%%%%%%%%%%%%
\begin{rema}
We may always find a Riemann surface structure of $\Sigma$ such that $G$ is a group of conformal automorphisms. In fact, let us consider a Riemann surface structure on the surface $\Sigma/G$, say induced by a Fuchsian group $\Lambda$ of the first kind; in other words, $\Sigma/G$ is homeomorphic to ${\mathbb H}^{2}/\Lambda$. 
By lifting such a Riemann surface structure to $\Sigma$, under $Q$, we obtain a Riemann surface structure on $\Sigma$ for which $G$ is a group of conformal automorphisms and $Q$ is a holomorphic covering map.
Such a Riemann structure on $\Sigma$ corresponds to a Fuchsian group $F \leq {\rm PSL}_{2}({\mathbb R})$, with $F \lhd \Lambda$ and $G=\Lambda/F \leq {\rm Aut}(\Sigma_{F})$, where $\Sigma_{F}={\mathbb H}^{2}/F$ (which is homeomorphic to $\Sigma$), and $\Sigma_{F}/G$ is homeomorphic to $\Sigma/G$. The Fuchsian group $F$ is of the first kind as $\Lambda$ is of the first kind, and $F$ is a non-trivial normal subgroup.
\end{rema}

Let us consider Fuchsian groups, of the first kind, $F$ and $\Lambda$, such that $F \lhd \Lambda$, $\Lambda/F \cong G$, 
$\Sigma$ is homeomorphic to $\Sigma_{F}={\mathbb H}^{2}/F$, $\Sigma/G$ is homeomorphic to ${\mathbb H}^{2}/\Lambda$, and with $Q$ induced by the  inclusion of $F$ in $\Lambda$. We identify $G$ with $\Lambda/F$.

Let $\pi_{F}:{\mathbb H}^{2} \to \Sigma_{F}$ be a Galois covering with ${\rm Deck}(\pi_{F})=F$.

%%%%%%%%%%%%%
\subsection{Lifting property}
Assume $\Gamma$ is an infinite-type Schottky group, with $\Gamma$-invariant component  $\Omega$, such that there is  a conformal homeomorphism $\varphi:\Sigma_{F} \to \Omega/\Gamma$ (the existence of $\Gamma$ is ensured by Theorem \ref{teoretro}). Let $P:\Omega \to \Omega/\Gamma$ be a Galois covering with ${\rm Deck}(P)=\Gamma$ (i.e., a Schottky uniformization). Let us consider the group $\varphi G \varphi^{-1} \leq {\rm Aut}(\Omega/\Gamma)$. 

We say that $G$ {\it lifts} with respect to the Schottky uniformization $P$, if for every $f \in G$ there is some $\widehat{f} \in {\rm Aut}(\Omega)$ such that 
$P \circ \widehat{f}= (\phi \circ f \circ \phi^{-1}) \circ P.$

\begin{rema}
As seen in the proof of Theorem \ref{teoretro}, we may assume that ${\rm Aut}(\Omega) < {\rm PSL}_{2}({\mathbb C})$. In particular, the liftings $\widehat{f}$ are M\"obius transformations.
\end{rema}

\begin{rema}
If ${\mathfrak C}$ is a Schottky admissible configuration that defines $\Gamma$, then ${\mathcal F}=\varphi^{-1}(P({\mathfrak C}))$ is a Schottky system of loops on $\Sigma_{F}$. Conversely, as seen in the construction in Section \ref{Sec:pruebateoretro}, each Schottky system of loops ${\mathcal F}$ on $\Sigma_{F}$ can be obtained in that way for $\Gamma=\Gamma_{\mathcal F}$.
\end{rema}

%%%%%%%%%%%%%%%%
\subsection{Finding conditions for the lifting}
Let us consider a Schottky system of loops ${\mathcal F}$ on $\Sigma_{F}$ and its associated $(\Sigma,{\mathcal F})$-Schottky group $\Gamma_{\mathcal F}$.
Without loss of generality, may assume that each loop in ${\mathcal F}$ is a simple closed geodesic of $\Sigma_{F}$. 

Let $\Omega$ be the $\Gamma_{\mathcal F}$-invariant connected component of $\Omega(\Gamma_{\mathcal F})$, and $P:\Omega \to \Sigma_{F}$ be a Galois covering with ${\rm Deck}(P)=\Gamma_{\mathcal F}$.

Following the previous notations from Section \ref{Sec:pruebateoretro}, we set $N_{\mathcal F}$ as the normal subgroup of $F$ induced by ${\mathcal F}$ and 
let $\pi_{\mathcal F}:{\mathbb H}^{2} \to \Omega$ be a Galois covering with ${\rm Deck}(\pi_{\mathcal F})=N_{\mathcal F}$.

We have the inclusions
$$N_{\mathcal F} \lhd F \lhd \Lambda, \; G=\Lambda/F, \; \Gamma_{\mathcal F}=F/N_{\mathcal F}.$$

Let us observe that $G$ lifts with respect to the Schottky uniformization $P$ if and only if $N_{\mathcal F}$ is a normal subgroup of $\Lambda$. This asserts the following fact.

\begin{prop}
If the Schottky system of loops ${\mathcal F}$ is $G$-invariant, then $N_{\mathcal F} \lhd \Lambda$, so $G$ lifts with respect to $P$.
\end{prop}

\medskip
\noindent
{\bf Question 2:} {\it Is it true the converse?}

\medskip

In the following, we provide an affirmative answer to the above in the particular case that $S/G$ is of finite type.

\begin{theo}\label{teolifting}
Let us assume that $\Sigma_{F}/G$ is of finite type. Then $G$ lifts with respect to $P:\Omega \to \Sigma_{F}$ with ${\rm Deck}(\Gamma_{\mathcal F})$ if and only if there exists 
is a $G$-invariant Schottky system of loops $\widehat{\mathcal F}$ defining $N_{\mathcal F}$.
\end{theo}
\begin{proof}
There exists $L_{0}>0$ such that any two simple closed geodesics of length less than or equal to $L_{0}$ are disjoint. 

Our hypothesis on the quotient $\Sigma_{F}/G$ ensures that the length spectrum of simple closed geodesics of length bigger than $L_{0}$ is discrete (i.e., such a spectrum is an increasing sequence $l_{1} <l_{2}<\ldots$ that only accumulates at $\infty$.

Let us consider a maximal subcollection ${\mathcal A}_{0}$ of simple closed geodesics of length at most $L_{0}$ that lift to loops on $\Omega$ under $P$.
Such a collection is $G$-invariant.

Next, in the complement $\Sigma_{F} \setminus {\mathcal A}_{0}$, we take any simple closed geodesic $\alpha$ of minimal length bigger than $L_{0}$ that lifts to loops on $\Omega$ under $P$.
 If there is some $T \in G$ such that $T(\alpha) \cap \alpha \neq \emptyset$, then we can find a shorter geodesic with the same property. This is a contradiction. 

We may proceed inductively to obtain $\widehat{\mathcal F}$.

\end{proof}

\medskip
\noindent
{\bf Question 3:} With respect to the previous results:
\begin{enumerate}[leftmargin=15pt]
\item {\it Can we delete the condition for $\Sigma_{F}/G$ to be of finite type?}

\item {\it Can we assume $\Gamma$ to be an infinite-type Schottky group with a connected region of discontinuity?}

\item {\it Is there an infinite-type Schottky group $\Gamma$ with connected region of discontinuity such that $\Sigma$ is homeomorphic to $\Omega(\Gamma)/\Gamma$ and $M_{\Gamma}$ is an infinite handlebody such that the group $G$ lifts to $\Omega(\Gamma)$?}

\item {\it Is there a Schottky system of loops ${\mathcal F}$ on $\Sigma$ such that the group $G$ extends as a group of orientation-preserving homeomorphisms of $M_{(\Sigma,{\mathcal F})}$?}

\end{enumerate}

\medskip

%%%%%%%%%%%%%%%%%
%%%%%%%%%%%%%%%%%
\section{An Example: the Loch Ness Monster}

%%%%%%%%%%%%%%
\subsection{Non-lifting examples}
Let $S={\mathbb H}^{2}/K$ be a closed Riemann surface of genus two. Its homology cover $\Sigma=\widetilde{S}={\mathbb H}^{2}/K'$, where $K'$ is the commutator subgroup of $K$, is homeomorphic to the Loch Ness monster (LNM). The Riemann surface $\widetilde{S}$ admits a group $G \cong {\mathbb Z}^{4}$ as a group of conformal automorphisms such that $S=\widetilde{S}/G$.

In this case, by Theorem \ref{teolifting}, the extension of $G$ to a $\Sigma$-handlebody is equivalent to the existence of a $G$-invariant Schottky system of loops on $\Sigma$. The following provides a negative answer.

\begin{claim}
There is no $G$-invariant Schottky system of loops on $\widetilde{S}$.
\end{claim}
\begin{proof}
Let us assume there is a Schottky system of loops (geodesics) ${\mathcal F}$ on $\widetilde{S}$ which is $G$-invariant. The collection ${\mathcal F}$ will project to a finite collection ${\mathcal L}$ of pairwise disjoint simple loops on $S$, each one homotopically non-trivial and pairwise non-parallel, such that each connected component of $S \setminus {\mathcal L}$ is planar.
As we are considering the homology covering of $S$, each loop $\alpha \in {\mathcal F}$ projects to a simple loop $\beta \in {\mathcal L}$ that is dividing and homotopically non-trivial. So, $S \setminus \beta$ consists of two genus one surfaces, each with one boundary corresponding to $\beta$. This provides a contradiction to the above property on ${\mathcal L}$.
\end{proof}

The above result is in contrast to the case of free actions of finite abelian groups on closed surfaces.

%%%%%%%%%%%%%%%%
\subsection{Example of two topologically non-equivalent Schottky uniformizations of LNM}
\subsubsection{Construction 1}
Let us consider the following Schottky admissible configuration of circles 
$${\mathfrak C}_{1}=\left\{C_{j}=\{z \in {\mathbb C}: |z-j|=1/4\},C'_{j}=\{z \in {\mathbb C}: |z-(j+i)|=1/4\}; j \in {\mathbb N}\right\}.$$
In this case, ${\mathfrak C}_{1}'=\{\infty\}$.

Let $\tau(z)=\overline{z}+i$ (the reflection on the line ${\rm Im}(z)=1/2$) and $\tau_{j}$ be the reflection on the circle $C_{j}$.

If $A_{j}=\tau \circ \tau_{j}$, then $\Gamma_{1}=\langle A_{1},A_{2},\ldots\rangle$ defines an infinite-type Schottky group, with a connected region of discontinuity, such that $\Omega(\Gamma_{1})/\Gamma_{1}$ is homeomorphic to LNM.

\subsubsection{Construction 2}
Let $C=\{z \in {\mathbb C}: |z-2|=1/4\}$ and $T(z)=3z$.

Let us consider the following Schottky admissible configuration of circles 
$${\mathfrak C}_{2}=\left\{C_{j}=T^{j}(C), C'_{j}=T^{-j}(C); j \in {\mathbb N}\right\},$$
In this case ${\mathfrak C}_{2}'=\{0,\infty\}$.

If $B_{j}$ is a loxodromic transformation such that $B_{j}(C_{j})=C'_{j}$ and $B_{j}({\rm Int}(C_{j}))={\rm Ext}(C'_{j})$, then $\Gamma_{2}=\langle B_{1},B_{2},\ldots\rangle$ defines an infinite-type Schottky group, with a connected region of discontinuity, such that $\Omega(\Gamma_{2})/\Gamma_{2}$ is homeomorphic to LNM.

%%%%%%%%%%%%%%%%%%%%%%%%%%
%%%%%%%%%%%%%%%%%%%%%%%%%%

%%%%%%%%%%%%%%%%%%%%%%
%\medskip
%\noindent{\bf Acknowledgements.} 
%The author would like to express his deep gratitude to the referee for supplying very useful comments, suggestions and corrections. 

%\medskip

%\noindent{\bf Funding.} The author was partially supported by Projects ANID FONDECYT Regular 1230001. 

%\medskip

%\noindent{\bf Data Availability Statement.} No data has been used in this article.

%\medskip

%\noindent{\sc Declarations.}

%\medskip

%\noindent{\bf Conflicts of interest.} The author has no conflicts of interest to declare that are relevant to this article

%\medskip

%\noindent{\bf Declaration of generative AI and AI-assisted technologies in the manuscript preparation process}
%During the preparation of this work, the authors used Grammarly to improve the English language. The author reviewed and edited the content as needed and take full responsibility for the content of the published article.

%%%%%%%%%%%%%%%%%%%%%%%%%%%%%%
%%%%%%%%%%%%%%%%%%%%%%%%%%%%%%

\end{document}